\documentclass{amsart}
\title[]{Realizations of free actions \\ via their fixed point algebras}
\date{\today}

\usepackage{amsmath, amsthm, amssymb, bbm, mathtools, nicefrac}
\usepackage{xifthen}
\usepackage[english]{babel}
\usepackage[T1]{fontenc}
\usepackage[utf8]{inputenc}
\usepackage{lmodern}
\usepackage{tabularx}
\usepackage{marvosym}
\usepackage{wasysym}
\usepackage{graphicx}
\usepackage[bookmarks,bookmarksdepth=1]{hyperref}
\usepackage[dvipsnames]{xcolor}
\usepackage{eurosym}
\usepackage{microtype}		
\usepackage[english]{babel}
\usepackage{enumitem}
	\newlist{equivalence}{enumerate}{1}
	\setlist[equivalence]{label=(\alph*)}
\usepackage{xspace}
\usepackage[final]{pdfpages}
\usepackage{url}
\usepackage{tikz}
\usepackage{tabu}
\usetikzlibrary{topaths}
\tikzstyle{every picture}+=[remember picture,inner xsep=0,inner ysep=0.25ex]
\usepackage{booktabs}
\usepackage{fixme}
\usepackage{a4wide}
\usepackage{aligned-overset}

\theoremstyle{plain}
	\newtheorem{theorem}{Theorem}[section]
	\newtheorem{lemma}[theorem]{Lemma}
	\newtheorem{corollary}[theorem]{Corollary}
	
\theoremstyle{definition}
	\newtheorem{defn}[theorem]{Definition}
\theoremstyle{remark}
	\newtheorem{remark}[theorem]{Remark}

\makeatletter
\newcommand{\setword}[2]{%
  \phantomsection
  #1\def\@currentlabel{\unexpanded{#1}}\label{#2}%
}

\makeatletter
\newcommand{\raisemath}[1]{\mathpalette{\raisem@th{#1}}}
\newcommand{\raisem@th}[3]{\raisebox{#1}{$#2#3$}}
\makeatother

\newcommand*{\C}{\mathbb{C}}		
\newcommand*{\tensor}{\otimes}		

\DeclarePairedDelimiter{\scal}{\langle}{\rangle}	

\DeclareMathOperator{\id}{id}		
\DeclareMathOperator{\Ad}{Ad}		

\DeclareMathOperator{\Aut}{Aut}

\newcommand{\cf}{\mbox{cf.}\xspace}				
\newcommand{\eg}{\mbox{e.\,g.}\xspace}			
\newcommand*{\ie}{\mbox{i.\,e.}\xspace}			
\newcommand{\wilog}{\mbox{w.\,l.\,o.\,g.}\xspace}	
\newcommand{\wrt}{\mbox{w.\,r.\,t.}\xspace}		
\newcommand*{\Star}{$^*$\nobreakdash}

\newcommand*{\hilb}{\mathfrak}		
\newcommand*{\alg}{\mathcal}		
\newcommand*{\hH}{\hilb H}			
\newcommand*{\one}{1}				
\newcommand*{\aA}{\alg{A}}			
\newcommand*{\aB}{\alg{B}} 			
\newcommand*{\aC}{\alg{C}} 			
\newcommand*{\aD}{\alg{D}} 			
\newcommand*{\End}{\mathcal L}	
\newcommand*{\Com}{\mathcal K}	
	
\newcommand*{\Mul}{\mathcal M}
\newcommand*{\Cont}{C}

\newcommand*{\fix}{\text{Fix}}
\DeclareMathOperator{\Irrep}{Irr} 			
\newcommand*{\acts}{\,.\,}
\DeclarePairedDelimiterXPP{\lprod}[2]
	{\,_{#1}}						
	{\langle}						
	{\rangle}						
	{}								
	{#2}							
\DeclarePairedDelimiterX{\rprod}[2]{\langle}{\rangle_{#1}}{#2}
\DeclarePairedDelimiterX{\ketbra}[2]{\lvert}{\rvert}{#1 \delimsize\rangle \delimsize\langle #2}
\DeclarePairedDelimiterX{\braket}[2]{\langle}{\rangle}{#1 \delimsize\vert #2}
\DeclarePairedDelimiterX{\matrixel}[3]{\langle}{\rangle}{#1 \delimsize\vert #2 \delimsize\vert #3}

\newcommand{\doubleitem}{%
  \begingroup
  \stepcounter{enumi}%
  \edef\tmp{\theenumi,}%
  \stepcounter{enumi}
  \edef\tmp{\endgroup\noexpand\item[\tmp\labelenumi]}%
  \tmp}
  
 \newcommand{\tripleitem}{%
 \begingroup \stepcounter{enumi}%
 \edef\tmp{\theenumi,}%
 \stepcounter{enumi} 
 \edef\tmpt{\theenumi,} 
 \stepcounter{enumi} 
 \edef\tmp{\endgroup\noexpand\item[\tmp\tmpt\labelenumi]}%
 \tmp}   
 
\makeatletter
\@namedef{subjclassname@2020}{%
  \textup{2020} Mathematics Subject Classification}
\makeatother

\hypersetup{colorlinks=true,
  linkcolor = black,
  citecolor = black, 
  urlcolor =  black, 
}

\begin{document}

\author{Kay Schwieger}
\address{iteratec GmbH, Stuttgart}
\email{kay.schwieger@gmail.com}

\author{Stefan Wagner}
\address{Blekinge Tekniska H\"ogskola}
\email{stefan.wagner@bth.se}



\begin{abstract}
Let $G$ be a compact group,  let $\aB$ be a unital C\Star-algebra, and let $(\aA,G,\alpha)$ be a free C\Star-dynamical system, in the sense of Ellwood, with fixed point algebra $\aB$.
We prove that $(\aA,G,\alpha)$ can be realized as the invariants of an equivariant coaction of $G$ on a corner of $\aB \tensor \Com(\hH)$ for a certain Hilbert space $\hH$ that arises from the freeness of the action.
This extends a result by Wassermann~\cite{Wass88a} for free C\Star-dynamical systems with trivial fixed point algebras.
As an application, we show that any faithful \Star-representation of $\aB$ on a Hilbert space $\hH_\aB$ gives rise to a faithful covariant representation of $(\aA,G,\alpha)$ on some truncation of $\hH_\aB \otimes \hH$.
\end{abstract}

\maketitle

\section{Introduction}

Let $\aA$ be a unital C\Star-algebra and let $G$ be a compact group that acts strongly continuously on $\aA$ by \Star-automorphisms $\alpha_g:\aA \to \aA$, $g \in G$.
In this article, we call this data a \emph{C\Star-dynamical system}, denote it briefly by $(\aA,G,\alpha)$, and customarily write $\aB$ for its fixed point algebra. 
Research into C\Star-dynamical systems is inherently interesting and has always been an area of active research both in operator algebras and noncommutative geometry.
It is desirable to identify properties of C\Star-dynamical systems that are both commonly occurring and significant enough to obtain interesting results.
To expedite matters, let us revisit the notion of \emph{freeness}, which exemplifies one such property:
A C\Star-dynamical system $(\aA,G,\alpha)$ is called \emph{free} if the so-called \emph{Ellwood map} 
\begin{align*}
	\Phi: \aA \tensor_{\text{alg}} \aA \rightarrow C(G,\aA),
	&&
	\Phi(x\tensor y)(g):=x \alpha_g(y)
\end{align*}
has dense range with respect to the canonical C\Star-norm on $\Cont(G,\aA)$.
This condition, first introduced by Ellwood~\cite{Ell00} for actions of quantum groups on C\Star-algebras, is known to be equivalent to Rieffel's saturatedness~\cite{Rieffel91} and the Peter-Weyl-Galois condition~\cite{BaCoHa15}.
By \cite[Prop.~7.1.12 and Thm.~7.2.6]{Phi87},  a continuous action $r:P\times G\rightarrow P$ of a compact group $G$ on a compact space $P$ is free in the classical sense, \ie, all stabilizer groups are trivial, if and only if the induced C\Star-dynamical system $(C(P),G,\alpha)$,  where $\alpha_g(f)(p) := f(r(p,g))$ for all $g\in G$, $f\in C(P)$,  and $p\in P$, is free in the sense of Ellwood. 
Free C\Star-dynamical systems thus provide a natural framework for noncommutative principal bundles.
Because of this and their wide range of applications, these objects have garnered widespread interest and have been extensively studied by many researchers in recent years
(see, \eg, \cite{AnGuIsRu22,BaCoHa15,CacMes19,DaSiZu14,FaLa24, GoLaPe94,Phi09,SchWa17,SchWa20} and ref.~therein).

Free and ergodic C\Star-dynamical systems, also known as full multiplicity ergodic actions,  have been the focal point of~\cite{Wass88a}.
In particular, Wassermann proved the interesting result that each free and ergodic C\Star-dynamical system $(\aA,G,\alpha)$ can be realized as the invariants of an equivariant coaction of $G$ on $\Com((L^2(G))$. 
Noteworthily, this constitutes an important step in the classification of such C\Star-dynamical systems by means of a generalized cocycle theory.

Now,  let us consider a free C\Star-dynamical  system $(\aA,G,\alpha)$ with a general fixed point algebra~$\aB$.
The overall purpose of this paper is to extend Wassermann's result by showing that $(\aA,G,\alpha)$ can be realized as the invariants of an equivariant coaction of $G$ on a corner of $\aB \tensor \Com(\hH)$ for a certain Hilbert space $\hH$ that arises from the freeness property (Theorem~\ref{thm:pi_S}).
From this we derive the following representation theoretic result: 
Any faithful \Star-representation of $\aB$ on a Hilbert space $\hH_\aB$ gives rise to a faithful covariant representation of $(\aA,G,\alpha)$ on some truncation of $\hH_\aB \otimes \hH$ (Corollary~\ref{cor:pi_S}).
It is our hope that Theorem~\ref{thm:pi_S} can be further utilized to present conditions under which certain properties of the fixed point algebra $\aB$ pass over to $(\aA,G,\alpha)$. 
Note that we do not address the problem of classifying (general) free C\Star-dynamical  systems in this paper.
In fact,  this problem has been thoroughly treated in~\cite{SchWa17} using different methods.

Following this introduction, the fundamental definitions and notations are presented in Section~\ref{sec:pre+not}. 
The proofs are provided in Section~\ref{sec:free_co}, which is in essence divided into three parts: first, constructing an equivariant coaction from a given free C\Star-dynamical system; second, identifying the corresponding invariants as a free C\Star-dynamical system; and third, proving that the original and the derived free C\Star-dynamical systems are isomorphic.

\section{Preliminaries and notations}\label{sec:pre+not}


\subsection{About tensor products}

Throughout this article tensor products of C\Star-algebras are taken with respect to the minimal tensor product, which is simply denoted by $\tensor$. 
Let $\aA$, $\aB$, and $\alg C$ be unital C\Star-algebras. 
If there is no ambiguity, then we consider each one of them as a C\Star-subalgebra of $\aA \tensor \aB \tensor \alg C$ and extend maps on $\aA$, $\aB$, or $\alg C$ canonically by tensoring with the respective identity map. 
For the sake of clarity, we also make use of the leg numbering notation, for instance, given $x \in \aA \tensor \alg C$, we write $x_{13}$ to denote the corresponding element in~$\aA \tensor \aB \tensor \alg C$.

\subsection{About multiplier algebras}\label{sec:multipliers}

Let $\aA$ be a C\Star-algebra. 
A linear operator $m: \aA \to \aA$ is said to be a multiplier of $\aA$ if for each $a \in \aA$ there exists $c \in \aA$ such that $a^* m(b) = c^* b$ for all $b \in \aA$. 
The set of all multipliers of $\aA$ is a unital C\Star-algebra which is denoted by $\Mul(\aA)$ and called the multiplier algebra of $\aA$.
For a faithful and nondegenerate \Star-representation $\pi: \aA \to \End(\hH)$ it may be identified with the following operator algebra:
\begin{equation*}
	\Mul(\aA) = \{ m \in \End(\hH) : (\forall a \in \aA) \, ma, am \in \aA\}.
\end{equation*}
For a thorough treatment of multipliers we refer to~\cite[Sec. 3.12]{Ped18}.


\begin{lemma}\label{lem:multipliers}
Let $\pi_\aA : \aA \to \End(\hH_\aA)$ and $\pi_\aB : \aB \to \End(\hH_\aB)$ be faithful and nondegenerated \Star-homomorphisms of C\Star-algebras $\aA$ and $\aB$ respectively. 
Then
\begin{equation*}
	\Mul(\aA) \tensor \aB = (\End(\hH_\aA) \tensor \aB) \cap \Mul(\aA \tensor \aB).
\end{equation*}
\end{lemma}
\begin{proof}
Clearly, $\Mul(\aA) \tensor \aB \subseteq  (\End(\hH_\aA) \tensor \aB) \cap \Mul(\aA \tensor \aB)$. 
Hence it suffices to establish the opposite inclusion.
For this purpose, we first assume that $\aB$ is unital.
Let $x := \sum_{i=1}^n x_i \tensor b_i \in \Mul(\aA \tensor \aB)$ for operators
$x_1,\ldots,x_n \in \End(\hH_\aA)$ and linear independent elements $b_1,\ldots b_n \in \aB$.
Moreover, let $\psi$ be a state on $\aB$ such that $\psi(b_i^* b_i) > 0$ for all $1 \le i \le n$. 
Applying Gram-Schmidt, we may \wilog assume that $\psi(b_i^* b_j) = \delta_{i,j}$ for all $1 \le i,j \le n$, \ie, $b_1, \dots, b_n$ is an orthonormal system.
We write $\omega_i(b) := \psi(b_i^* b)$, $b \in \aB$, $1 \le i \le n$, for the corresponding dual system of linear functionals on $\aB$. 
Since $x \in \Mul(\aA \tensor \aB)$, for each $a \in \aA$ we have $x(a \tensor \one_\aB) \in \aA \tensor \aB$. 
Consequently,
\begin{equation*}
	x_i a 
	= (\id \tensor \omega_i(x)) a 
	= \id \tensor \omega_i (x (a \tensor \one_\aB)) \in \aA
\end{equation*}
for all $1 \le i \le n$ and $a \in \aA$.
Likewise, we find $ax_i \in \aA$. 
It follows that $x_i \in \Mul(\aA)$ for all $1 \le i \le n$, and hence $x \in \Mul(\aA) \tensor \aB$. 
Taking limits, we thus get $\bigl( \End(\hH_A) \tensor \aB \bigr) \cap \Mul(\aA \tensor \aB) \subseteq \Mul(\aA) \tensor \aB$. 
For non-unital $\aB$ we may replace $\one_\aB$ by an approximate unit of $\aB$ and use similar arguments, the detailed verification being left to the reader. 
\end{proof}


\subsection{About coactions of compact groups}
\label{sec:coaction}

Let $G$ be a compact group. 
Here and subsequently, we denote by $\lambda : G \to \alg U (L^2(G))$ the left regular representation given by $(\lambda_gf)(h):=f(g^{-1}h)$ and by $r : G \to \alg U (L^2(G))$ the right regular representation given by $(r_gf)(h):=f(hg)$.
For $f \in C(G)$ we write $r(f)$ for the integrated form with respect to the right regular representation and consider $C^*_r(G)$ as the norm closure of the \Star-algebra $r(C(G)) \subseteq \End(L^2(G))$ or, equivalently, as the fixed point algebra of $\Com(L^2(G))$ under the adjoint action $\Ad[\lambda_g]$, $g\in G$. 

Usually, we consider $C^*_r(G)$ as a quantum group with respect to the faithful and nondegenerate \Star-homomorphism $\delta_G: C^*_r(G) \to \Mul(C^*_r(G) \tensor C^*_r(G))$
defined by the integrated form of the diagonal representation $G \ni g \mapsto r_g \tensor r_g \in \alg U(L^2(G) \tensor L^2(G))$. 
Due to~\cite[pp. 255]{Lan79}, the unitary $W_G$ on $L^2(G \times G) \cong L^2(G) \tensor L^2(G)$ defined by $(W_G f)(g,h) := f(g,hg)$
implements $\delta_G$ in the sense that $\delta_G(x) = W_G^* \bigl( x \tensor \one_G \bigr) W_G$ for all $x \in C^*_r(G)$. 
We also utilize the fact that $W_G \in \Mul \bigl( C(G) \tensor C^*_r(G) \bigr)$ as well as the identities
\begin{align}
	(W_G)_{23}(W_G)_{12}(W_G)_{13} &= (W_G)_{12}(W_G)_{23},
	\label{eq:WGcocycel}
	\\
	(r_g \tensor \one_G) W_G &= W_G (r_g \tensor r_g)
	&&
	\forall g \in G,
	\label{eq:WGright}
	\\
	(\one_G \tensor \lambda_g) W_G &= W_G (\one_G \tensor \lambda_g) 
	&&
	\forall g \in G,
	\label{eq:WGequivariance}
	\\
	(\lambda_g \tensor r_g) W_G &= W_G (\lambda_g \tensor \one_G)
	&&
	\forall g \in G.
	\label{eq:WGcommuting}
\end{align}

More generally, we are concerned with coactions of compact groups on C\Star-algebras and refer to~\cite[App.~A]{EKQR06} for a detailed discussion on the subject. 
However, for expediency we now repeat the basic definition.

\begin{defn}
A \emph{coaction} of a compact group $G$ on a C\Star-algebra $\aA$ is a faithful and nondegenerate \Star-homomorphism $\delta: \aA \to \Mul(\aA \tensor C^*_r(G))$ satisfying the \emph{coaction identity} 
\begin{equation}
	(\delta \tensor \id) \circ \delta = (\id \tensor \delta_G) \circ \delta.\label{eq:coactcon}
\end{equation}
As coactions take values in multiplier algebras, the coaction identity involves the extensions of the maps $\delta \tensor \id$ and $\id \tensor \delta_G$ to the multiplier algebras $\Mul(\aA \tensor C^*_r(G))$ and $\Mul(\aA \tensor C^*_r(G) \tensor C^*_r(G))$ in domain and codomain, respectively. 
\end{defn}

\subsection{About free C$^*$-dynamical systems}\label{sec:C*}

One of the key tools used in this article is a characterization of freeness that we provided in \cite[Lem.~3.2]{SchWa17}, namely that a C\Star-dynamical system $(\aA,G,\alpha)$ is free if and only if for each irreducible representation $(\sigma,V_\sigma)$ of $G$ there is a finite-dimensional Hilbert space~$\hH_\sigma$ and an isometry $s(\sigma) \in \aA \tensor \End(V_\sigma,\hH_\sigma)$ satisfying $\alpha_g(s(\sigma))=s(\sigma) (\one_\aA \tensor \sigma_g)$ for all $g \in G$. 
However, to simplify notation we patch this family of isometries together and use the following characterization instead. 

\begin{lemma}[\cf~{\cite[Lem. ~3.1]{SchWa20}}]\label{lem:coisometry}
	For a C\Star-dynamical system $(\aA, G, \alpha)$ the following statements are equivalent:
	\begin{equivalence}
	\item
		$(\aA, G, \alpha)$ is free.
	\item
		There is a unitary representation $\mu: G \to \mathcal U(\hH)$ with finite-dimensional multiplicity spaces and, given any faithful covariant representation $(\pi,u)$ of $(\aA,G,\alpha)$ on some Hilbert space~$\hH_\aA$, an isometry $S \in \End(\hH_\aA \tensor L^2(G), \hH_\aA \tensor \hH)$ satisfying
		\begin{align}
			S \aA \tensor \Com(L^2(G)) 
			&\subseteq 
			\aA \tensor \Com(L^2(G),\hH),\label{eq:SOPleft}
			\\
			(u_g \tensor \one_\hH) S 
			&= 
			S (u_g  \tensor r_g) 
			&&
			\forall g \in G,
			\label{eq:SOPequivariance}
			\\
			(\one_\aA \tensor \mu_g) S 
			&= 
			S (\one_\aA \tensor \lambda_g) 
			&&
			\forall g \in G.\label{eq:SOPcommuting}
		\end{align}
Here, we do not distinguish between $\aA$ and $\pi(\aA) \subseteq \End(\hH_\aA)$ for sake of brevity.
Furthermore, the tensor product $\aA \tensor \Com(L^2(G),\hH)$ is closed with respect to the operator norm, where $\Com(L^2(G),\hH)$ is regarded as the respective corner of $\Com(L^2(G) \oplus \hH)$.
	\end{equivalence}
\end{lemma}

\begin{remark}\label{rem:SOPright}
	The adjoint $S^* \in \End(\hH_\aA \tensor \hH, \hH_\aA \tensor L^2(G))$ of $S$ satisfies $S^* \aA \tensor \Com(\hH) \subseteq \aA \tensor \Com(\hH,L^2(G))$ or, equivalently, 
	\begin{equation}
		\aA \tensor \Com(\hH) S \subseteq \aA \tensor \Com(L^2(G),\hH).\label{eq:SOPright}
	\end{equation}
	For this reason, we can assert that $S$ is, in fact, a multiplier for $\aA \tensor \Com(L^2(G) \oplus \hH)$, \ie, $S \in \Mul(\aA \tensor \Com( L^2(G) \oplus \hH))$, with $(1_\aA \tensor p_{L^2(G)}) S = 0 = S (1_\aA \tensor p_\hH)$, where $p_\hH$ and $p_{L^2(G)}$ denote the canonical projections onto $\hH$ and $L^2(G)$ respectively.
\end{remark}

\begin{remark}\label{rem:cleft}
	A particular simple class of free actions is given by so-called cleft actions (see~\cite{SchWa16}). 
Regarded as noncommutative principal bundles, these actions are essentially characterized by the fact that all associated noncommutative vector bundles are trivial. 
For convenience of the reader we now recall the definition. 
Indeed, we call a C\Star-dynamical system $(\aA,G,\alpha)$ \emph{cleft} if there is a unitary $U \in \Mul(\aA \tensor C^*_r(G))$ such that $\bar{\alpha}_g(U) = U (1_\aA \tensor r_g)$ for all $g \in G$.
Here, $\bar{\alpha}_g$, $g\in G$, denotes the strictly continuous extension of $\alpha_g \tensor \one_G$, $g\in G$, to $\Mul(\aA \tensor C^*_r(G))$ (see \eg~\cite[Prop. 3.12.10]{Ped18}), which is continuous for the strict topology.  
It is clear that each cleft C\Star-dynamical system is free with a possible choice for $\mu$ and $\hH$ given by $\lambda$ and $L^2(G)$ respectively.
\end{remark}

\section{The realization}\label{sec:free_co}

For a start, we fix a free C\Star-dynamical system $(\aA,G,\alpha)$ with fixed point algebra~$\aB$ together with a faithful covariant representation $(\pi,u)$ thereof on some Hilbert space $\hH_\aA$.
By Lemma~\ref{lem:coisometry}, there is a unitary representation $\mu: G \to \mathcal U(\hH)$ with finite-dimensional multiplicity spaces and an isometry $S \in \End(\hH_\aA \tensor L^2(G), \hH_\aA \tensor \hH)$ satisfying~\eqref{eq:SOPleft}, \eqref{eq:SOPequivariance}, \eqref{eq:SOPcommuting}, and~\eqref{eq:SOPright}. 

It is immediate that $P := SS^* \in \End(\hH_\aA \tensor \hH)$ is a projection satisfying $(u_g \tensor \mu_h) P = P (u_g \tensor \mu_h)$ for all $g,h \in G$. 
Moreover, we have $P \in \Mul(\aB \tensor \Com(\hH))$, which is easily checked.
In what follows, the central object of interest is the corresponding corner
\begin{equation}
		\aD := P ( \aB \tensor \Com(\hH) ) P \subseteq \aB \tensor \Com(\hH) \cap P\End(\hH_\aA \tensor \hH)P.\label{eq:D}
\end{equation}
We see at once that $\aD$ is pointwise fixed by $\Ad[u_g \tensor \one_\hH]$, $g\in G$.
Furthermore,  $\Ad[\one_\aA \tensor \mu_g]$, $g\in G$, acts strongly continuous on $\aD$, \ie, $(\aD,G,\Ad[\one_\aA \tensor \mu])$ is a C\Star-dynamical system. 
Our first task is to construct a coaction of $G$ on $\aD$ in terms of $S$ and the unitary $W_G$ (\cf Section~\ref{sec:coaction}). 
We begin with a series of lemmas.

\begin{lemma}\label{lem:WS_elementary}
	For the element $W_S := S_{12} (W_G)_{23} S_{12}^*$ in $\End(\hH_\aA \tensor \hH \tensor L^2(G))$ the following assertions hold:
	\begin{enumerate}[label={\arabic*}.,ref=\ref{lem:WS_elementary}.{\textit{\arabic*}}.]
	\item
		$W_S$ is a partial isometry with initial and final projection $P \tensor \one_G$.\label{lem:WSpariso}
	\item
		$(u_g \tensor \mu_h \tensor r_h) W_S = W_S (u_g \tensor \mu_h \tensor r_g)$ for all $g,h \in G$.\label{lem:WSequiv}
	\item
		$W_S \in \Mul(\aA \tensor \Com(\hH) \tensor C^*_r(G))$.\footnote{Note that $\aA \tensor \Com(\hH) \tensor C^*_r(G) \subseteq \End(\hH_\aA \tensor \hH \tensor L^2(G))$ is nondegenerate.}\label{lem:WSmul}
	\end{enumerate}
\end{lemma}
\begin{proof}
\begin{enumerate}[label={\arabic*}.]
\item
	Since $S$ is an isometry, we see at once that the initial and final projections are given by $W_{S}^* W_S = P \tensor \one_G$ and $W_S W_S^* = P \tensor \one_G$ respectively.
\item
	Let $g,h \in G$. 
	Then an easy verification yields
	\begin{align*}
		(u_g \tensor \mu_h \tensor r_h) W_S 
		&=
		(u_g \tensor \mu_h \tensor r_h) S_{12} (W_G)_{23} S_{12}^*
		\\
		\overset{\eqref{eq:SOPequivariance},\eqref{eq:SOPcommuting}}&{=} 
		S_{12} (u_g \tensor r_g \lambda_h \tensor r_h) (W_G)_{23} S_{12}^*
		\\
		\overset{\eqref{eq:WGright},\eqref{eq:WGcommuting}}&{=} 
		S_{12} (W_G)_{23} (u_g \tensor r_g \lambda_h \tensor r_g) S_{12}^*
		\\
		\overset{\eqref{eq:SOPequivariance},\eqref{eq:SOPcommuting}}&{=}  
		S_{12} (W_G)_{23} S_{12}^* (u_g \tensor \mu_h \tensor r_g) = W_S (u_g \tensor \mu_h \tensor r_g).
	\end{align*}
\item
	From~\eqref{eq:SOPleft} we deduce that $S \aA \tensor \Com(\hH,L^2(G)) \subseteq \aA \tensor \Com(\hH)$. Combining this with~(\ref{eq:SOPleft}) and the fact that $W_G \in \End(L^2(G)) \tensor \End(L^2(G))$ yields $W_S \in \Mul(\aA \tensor \Com(\hH) \tensor \Com(L^2(G)))$. 
	But because $C^*_r(G)$ is the fixed point algebra of $\Com(L^2(G))$ with respect to $\Ad[\lambda_g]$, $g\in G$, and $W_G$ satisfies~(\ref{eq:WGequivariance}), it may further be concluded that $W_S \in \Mul(\aA \tensor \Com(\hH) \tensor C^*_r(G))$. 
	\qedhere
\end{enumerate}
\end{proof}

Here and subsequently,  we do not distinguish between the algebras $\End(\hH_\aA \tensor \hH \tensor L^2(G))$ and $\End(\hH_\aA \tensor \hH) \tensor \End(L^2(G))$ for simplicity of notation.

\begin{lemma}\label{lem:delta_elementary}
	For the map
	\begin{align*}
		\delta: \End(\hH_\aA \tensor \hH) \to \End(\hH_\aA \tensor \hH \tensor L^2(G)), 
		&&
		\delta(x) : = \Ad[W_S](x \tensor \one_G) := W_S (x \tensor \one_G) W_S^*
	\end{align*}	
	the following assertions hold:
	\begin{enumerate}[label={\arabic*}.,ref=\ref{lem:delta_elementary}.{\textit{\arabic*}}.]
	\item
		$\delta$ is equivariant \wrt $\Ad[u_g \tensor \mu_h]$, $g,h \in G$ on $\End(\hH_\aA \tensor \hH)$ and $\Ad[u_g \tensor \mu_h \tensor r_h]$, $g,h \in G$, on $\End(\hH_\aA \tensor \hH \tensor L^2(G))$.
		\label{lem:deltaequiv} 
	\item
		$\delta(x) \in \Mul(\aB \tensor \Com(\hH) \tensor C^*_r(G))$ for all $x \in \aB \tensor \Com(\hH)$.\label{lem:deltamul}
	\item
		$\delta$ satisfies the comultiplication identity $(\delta \tensor \id) \circ \delta = (\id \tensor \delta_G) \circ \delta$.
\label{lem:deltacomulti}
\end{enumerate}
\end{lemma}
\begin{proof}
\begin{enumerate}[label={\arabic*}.]
\item
	Let $g,h \in G$ and let $x \in \End(\hH_\aA \tensor \hH)$.
	Then 
	\begin{align*}
		\delta(\Ad[u_g \tensor \mu_h](x)) &= \Ad[W_S](\Ad[u_g \tensor \mu_h \tensor \one_G](x \tensor \one_G))
		\\
		&= \Ad[W_S (u_g \tensor \mu_h \tensor \one_G)](x \tensor \one_G)
		\\
		&= \Ad[W_S (u_g \tensor \mu_h \tensor r_g)](x \tensor \one_G)
		\\
		&= \Ad[(u_g \tensor \mu_h \tensor r_h ) W_S](x \tensor \one_G) 
		\\
		&= \Ad[u_g \tensor \mu_h \tensor r_h ](\Ad[W_S](x \tensor \one_G)) = \Ad[u_g \tensor \mu_h \tensor r_h] (\delta(x)),
	\end{align*}
	where we have used Lemma~\ref{lem:WSequiv}~for the third from last equality. 
\item
	Let $x \in \aB \tensor \Com(\hH)$ and let $y \in \aB \tensor \Com(\hH) \tensor C^*_r(G)$. 
	Applying Lemma~\ref{lem:WSmul}, we see that $\delta(x)y \in \aA \tensor \Com(\hH) \tensor C^*_r(G)$.
	Moreover, the equivariance of $\delta$ implies that
	\begin{align*}
		\Ad[u_g \tensor \one_\hH \tensor \one_G](\delta(x)y) &= \Ad[u_g \tensor \one_\hH \tensor \one_G](\delta(x)) \Ad[u_g \tensor \one_\hH \tensor \one_G](y)
		\\
		&= \delta(\Ad[u_g \tensor \one_\hH](x)) y = \delta(x)y
	\end{align*}
	for all $g \in G$, \ie, $\delta(x)y \in \aB \tensor \Com(\hH) \tensor C^*_r(G)$. 
	Likewise, we get $y\delta(x) \in \aB \tensor \Com(\hH) \tensor C^*_r(G)$. 
	It follows that $\delta(x) \in \Mul(\aB \tensor \Com(\hH) \tensor C^*_r(G))$ as required.
\item
	To demonstrate that $(\delta \tensor \id) \circ \delta = (\id \tensor \delta_G) \circ \delta$, we first evaluate both sides of the identity.
	Indeed, for each $x \in \End(\hH_\aA \tensor \hH)$ we have
	\begin{align*}
		\delta \tensor \id (\delta(x)) &= \delta \tensor \id (\Ad[W_S] (x \tensor \one_G))
		\\
		&= \Ad[(W_S)_{123}(W_S)_{124}](x \tensor \one_G) 
		\\
		&= \Ad[S_{12}(W_G)_{23}(W_G)_{24} S_{12}^*](x \tensor \one_G).
	\end{align*}
	Moreover, the right hand side of the identity becomes
	\begin{align*}
		\id \tensor \delta_G (\delta(x)) &= \id \tensor \delta_G (\Ad[W_S](x \tensor \one_G))
		\\
		&= \Ad[(W_G)_{34}^*(W_S)_{123}](x \tensor \one_G)
		\\
		&= \Ad[(W_G)_{34}^*(W_S)_{123}(W_G)_{34}](x \tensor \one_G) 
		\\
		&= \Ad[(W_G)_{34}^*S_{12}(W_G)_{23}S_{12}^*(W_G)_{34}](x \tensor \one_G) 
		\\
		&= \Ad[S_{12}(W_G)_{34}^*(W_G)_{23}(W_G)_{34}S_{12}^*](x \tensor \one_G).
	\end{align*}
	Comparing these two expressions, we now see that the claim will be proved once we show that $(W_G)_{12}(W_G)_{13} = (W_G)_{23}^*(W_G)_{12}(W_G)_{23}$.
	But this is clear from~(\ref{eq:WGcocycel}). 
	\qedhere
\end{enumerate}
\end{proof}

\begin{remark}
    We stress that the map $\delta$ is, in general, not a \Star-homomorphism.
\end{remark}

To proceed, we put $\hH_P := P (\hH_\aA \tensor \hH)$, identify $\End(\hH_P)$ with $P \End(\hH_\aA \tensor \hH) P$ by means of the map $P \End(\hH_\aA \tensor \hH) P \to \End(\hH_P)$, $x \mapsto x |_{\hH_P}^{\hH_P}$, and note that $\aD \subseteq \End(\hH_P)$ is nondegenerate.

\begin{lemma}\label{lem:deltaD}
	Restricting $\delta$ to $\alg D$ (\cf~(\ref{eq:D})) yields a map 
	\begin{equation}
		\delta_\alg D : \alg D \rightarrow \Mul(\alg D \tensor C^*_r(G)) \subseteq \End(\hH_P \tensor L^2(G)), 
	\label{eq:deltaD}
	\end{equation}
for which the following assertions hold:
	\begin{enumerate}[label={\arabic*}.,ref=\ref{lem:deltaD}.{\textit{\arabic*}}.]
	\item
		$\delta_\alg D$ is a \Star-homomorphism.\label{lem:deltaDhom}
	\item
		$\delta_\alg D$ is faithful.\label{lem:deltaDfaith}
	\item
		$\delta_\alg D$ is nondegenerate.\label{lem:deltaDnondeg}
	\item
		$\delta_\alg D$ satisfies the comultiplication identity $(\delta_\alg D \tensor \id) \circ \delta_\alg D = (\id \tensor \delta_G) \circ \delta_\alg D$.\label{lem:deltaDcomulti}
	\item
		$\delta_\alg D$ is equivariant \wrt $\Ad[\one_\aA \tensor \mu_g]$, $g\in G$, on $\alg D$ and  $\Ad[\one_\aA \tensor \mu_g \tensor r_g]$, $g\in G$, on $\Mul(\alg D \tensor C^*_r(G))$.\label{lem:deltaDequiv}
	\end{enumerate}
\end{lemma}
\begin{proof}
\begin{enumerate}[start=0,label={\arabic*}.]
\item
	We first show that $\delta_\alg D$ is well-defined. 
	For this purpose, let $x=PyP \in \aD$ for some $y \in \aB \tensor \Com(\hH)$ and let $z \in \alg D \tensor C^*_r(G) \subseteq \aB \tensor \Com(\hH) \tensor C^*_r(G)$. 
	By Lemma~\ref{lem:deltamul},  $\delta(y) z \in \aB \tensor \Com(\hH) \tensor C^*_r(G)$. 
	Moreover,  $\delta(x)z = (P \tensor 1_G) \delta(y)  (P \tensor 1_G) z = (P \tensor 1_G) \delta(y) z (P \tensor 1_G)$.
	Consequently, $\delta(x)z \in (P \tensor 1_G) \aB \tensor \Com(\hH) \tensor C^*_r(G) (P \tensor 1_G) = \alg D \tensor C^*_r(G)$. 
	Because the same conclusion can be drawn for the element $z\delta(x)$, it follows that $\delta(x) \in \Mul(\alg D \tensor C^*_r(G))$.
\item 
	It is immediate that $\delta_\alg D$ is a \Star-map.
	To see that $\delta_\alg D$ is an algebra map,  let $x,y \in \alg D$. 
	Then
	\begin{align*}
		\delta_{\alg D}(xy) &= W_S  (xy \tensor \one_G ) W_S^* = W_S  (xPy \tensor \one_G ) W_S^* 
		\\ 
		&= W_S (x \tensor \one_G) (P \tensor \one_G) (y \tensor \one_G ) W_S^*
		\\ 
		&=\underbrace{W_S (x \tensor \one_G) W_S^*}_{=\delta_{\alg D}(x)} \underbrace{W_S (y \tensor \one_G ) W_S^*}_{=\delta_{\alg D}(y)}.  
	\end{align*}
\item
	Let $x \in \alg D$ such that $\delta(x) =0$. 
	Multiplying the latter equation by $W_S^*$ from the left and by $W_S$ from the right gives $PxP \tensor \one_G = x \tensor \one_G = 0$.
	Hence $x=0$ as required.
\item
	Let $\{ T_i \}$ be a bounded approximate identity for $\Com(\hH)$.
	Then $\{\delta(P(\one_\aB \tensor T_i)P)\}$ is a bounded approximate identity for $\alg D \tensor C^*_r(G)$, \ie, $\delta(P(\one_\aB \tensor T_i)P)$ converges to $\one$ in the strict topology of $\Mul (\alg D \tensor C^*_r(G))$. 
	In particular, $\delta(\alg D) (\alg D \tensor C^*_r(G))$ is dense in $\alg D \tensor C^*_r(G)$ which amounts to saying that $\delta_\aD$ is nondegenerate. 
	As a matter of fact, we even have equality by the Cohen factorization theorem (see, \eg,~\cite[Prop. 2.33]{Rae98}).
\item
	Since $\delta_\aD$ is nondegenerate, $\delta_D \tensor \id : \aD \tensor C^*_r(G) \to \Mul(\aD \tensor C^*_r(G) \tensor C^*_r(G))$ uniquely extends to a \Star-homomorphisms with domain $\Mul(\aD \tensor C^*_r(G))$ which, by uniqueness, must agree with the restriction-corestriction
	\begin{equation*}
		\delta \tensor \id |_{\Mul(\aD \tensor C^*_r(G))}^{\Mul(\aD \tensor C^*_r(G) \tensor C^*_r(G))}.
	\end{equation*}
	This establishes the statement when combined with Lemma~\ref{lem:deltacomulti}
\item
	By Lemma~\ref{lem:deltaequiv}, it suffices to note that $\aD$ is invariant under $\Ad[\one_\aA \tensor \mu_g]$, $g\in G$, and that $\Mul(\aD \tensor C^*_r(G))$ is invariant under $\Ad[\one_\aA \tensor \mu_g \tensor r_g]$, $g\in G$.
	\qedhere
	\end{enumerate}
\end{proof}

\begin{remark}\label{rem:deltaD}
	Much the same proof as above also works for the restriction of $\delta$ to $\End(\hH_P)$ in domain and codomain, \ie, the map $\End(\hH_P) \to \End(\hH_P \tensor L^2(G))$, $x \mapsto \delta(x)$. 
\end{remark}

\begin{corollary}\label{cor:deltaD}
	$\delta_\aD$ is a coaction of $G$ on $\aD$ that is equivariant \wrt $\Ad[\one_\aA \tensor \mu_g]$, $g\in G$, on $\alg D$ and $\Ad[\one_\aA \tensor \mu_g \tensor r_g]$, $g\in G$, on $\Mul(\alg D \tensor C^*_r(G))$.
\end{corollary}

The task is now to recover $(\aA,G,\alpha)$ from~$\delta_\aD$. 
For this purpose, we recall that $\Mul(\aD)$ may be identified with $P \Mul(\aB \tensor \Com(\hH)) P$ (\cf~\cite[Sec. II.7.3.14]{BB12})
and consider $\delta_\aD$'s fixed point algebra:
\begin{equation*}
	\fix(\delta_\aD) := \{ x \in \Mul(\aD) : \delta_\aD(x) = x \tensor \one_G \} \subseteq \End(\hH_P).
\end{equation*} 
Since $P \in \Mul(\aD)$ and $\delta_\aD(P) = P \tensor \one_G$, we see that $\fix(\delta_\aD)$ is a unital C\Star-subalgebra of $\Mul(\aD)$.
Furthermore, $\fix(\delta_\aD)$ is invariant under $\Ad [\one_\aA \tensor \mu_g]$, $g\in G$, as is easy to check.
Unfortunately, the resulting action $G \to \Aut(\fix(\delta_\aD))$ is unlikely to be continuous for the norm topology on~$\Mul(\aD)$. The best we can do is say that it is continuous for the strict topology. 
This inconvenience is simplest avoided by only taking into account those elements $x \in \fix(\delta_\aD)$ for which the map $G \to \fix(\delta_\aD)$, $g \mapsto \Ad [\one_\aA \tensor \mu_g](x)$ is norm continuous, \ie, 
	\begin{equation*}
		\tilde{\aA} := \{ x \in \fix(\delta_\aD) : G \ni g \mapsto \Ad[\one_\aA \tensor \mu_g](x) \in \fix(\delta_\aD) \,\,\, \text{norm continuous} \}.
	\end{equation*}
Now, a straightforward verification shows that $\tilde{\aA}$ is a unital C\Star-subalgebra of $\fix(\delta_\aD)$ with unit $\one_{\tilde{\aA}} := P$ on which $\Ad [\one_\aA \tensor \mu_g]$, $g \in G$, acts strongly continuous. 
In particular, writing $\tilde{\alpha}_g$,~$g \in G$, for the restriction of $\Ad [\one_\aA \tensor \mu_g]$, $g \in G$, to $\tilde{\aA}$ in both domain and codomain, we conclude that $(\tilde{\aA},G,\tilde{\alpha})$ is a concrete C\Star-dynamical system on $\hH_P$.

\begin{lemma}\label{lem:reconstruction}
	For the element $\tilde{S} := S_{12} S_{13} (W_G)_{23}^* S_{12}^* \in \End(\hH_P \tensor L^2(G), \hH_P \tensor \hH)$
	the following assertions hold:
	\begin{enumerate}[label={\arabic*}.,ref=\ref{lem:reconstruction}.{\textit{\arabic*}}.]
	\item
		$\tilde{S}$ is an isometry satisfying $\delta \tensor \id (\tilde{S}) = \tilde{S}_{124}$.\label{lem:reconstructionCOISO}
	\item
		$(u_g \tensor \mu_h \tensor \mu_k) \tilde{S} = \tilde{S} (u_g \tensor \mu_h \tensor r_h \lambda_k)$ for all $g,h,k \in G$.\label{lem:reconstructionEQUiV}
	\item
		$\tilde{S} \aD \tensor \Com(L^2(G)) \subseteq \aD \tensor \Com(L^2(G),\hH)$ and $\aD \tensor \Com(\hH) \tilde{S} \subseteq \aD \tensor \Com(L^2(G),\hH)$, \ie, 
		\begin{equation*}
			\tilde{S} \in \Mul(\aD \tensor \Com(L^2(G) \oplus \hH))
		\end{equation*}
		with $(1_{\hH_P} \tensor p_{L^2(G)}) \tilde{S} = 0 = \tilde{S} (1_{\hH_P} \tensor p_{\hH})$, where $p_{\hH}$ and $p_{L^2(G)}$ denote the canonical projections onto $\hH$ and $L^2(G)$ respectively (\cf Remark~\ref{rem:SOPright}).\label{lem:reconstructionMULTI}
	\item
		$\tilde{S} \Mul(\aD) \tensor \Com(L^2(G)) \subseteq \Mul(\aD) \tensor \Com(L^2(G),\hH)$.\label{lem:reconstructionSUBMULTI}
		\end{enumerate}
\end{lemma}
\begin{proof}
\begin{enumerate}[label={\arabic*}.]\item
	We have $\tilde{S}^*\tilde{S} = S_{12} (W_G)_{23} S_{13}^* S_{12}^* S_{12} S_{13} (W_G)_{23}^* S_{12}^* = S_{12} S_{12}^* = \one_{\tilde{\aA}} \tensor \one_G$, \ie, $\tilde{S}$ is an isometry. 
	Furthermore, it is easily deduced that
	\begin{align*}
		\delta \tensor \id (\tilde{S}) &= \Ad[ S_{12} (W_G)_{23} S_{12}^*] (\tilde{S}_{124}) 
		\\
		&=S_{12} (W_G)_{23} S_{12}^* S_{12} S_{14} (W_G)_{24}^* S_{12}^* S_{12} (W_G)_{23}^* S_{12}^*				
		\\
		&= S_{12} S_{14} (W_G)_{23} (W_G)_{24}^* (W_G)_{23}^* S_{12}^*
		\\
		&= S_{12} S_{14} (W_G)_{24}^* S_{12}^* = \tilde{S}_{124},
	\end{align*}
	where the second last equality is due to the fact that $W_G \in L^\infty(G) \tensor \End(L^2(G))$. 
\item
	Let $g,h,k \in G$. 
	Then 
	\begin{align*}
		(u_g \tensor \mu_h  \tensor \mu_k) \tilde{S} &= (u_g \tensor \mu_h  \tensor \mu_k) S_{12} S_{13} (W_G)_{23}^* S_{12}^*
		\\
		\overset{\eqref{eq:SOPequivariance},\eqref{eq:SOPcommuting}}&{=}
		S_{12} S_{13} (u_g \tensor r_g \lambda_h  \tensor r_g \lambda_k) (W_G)_{23}^* S_{12}^*
		\\
		\overset{\eqref{eq:WGright},\eqref{eq:WGequivariance}, \eqref{eq:WGcommuting}}&{=}
		S_{12} S_{13} (W_G)_{23}^* (u_g \tensor r_g \lambda_h  \tensor r_h \lambda_k) S_{12}^*
		\\
		\overset{\eqref{eq:SOPequivariance},\eqref{eq:SOPcommuting}}&{=}
		S_{12} S_{13} (W_G)_{23}^* S_{12}^* (u_g \tensor \mu_h  \tensor r_h \lambda_k) = \tilde{S} (u_g \tensor \mu_h  \tensor r_h \lambda_k).
	\end{align*}
\item
	Using successively the identity $\tilde{S} = S_{12} S_{13} S_{12}^* W_S^*$, the inclusion $\aD \subseteq \aA \tensor \Com(\hH)$, Lemma~\ref{lem:WSmul}, and three times~\eqref{eq:SOPleft} yields
	\begin{align*}
		\tilde{S} \aD \tensor \Com(L^2(G)) &= S_{12} S_{13} S_{12}^* W_S\aD \tensor \Com(L^2(G)) 
		\\
		&\subseteq S_{12} S_{13} S_{12}^*  \aA \tensor \Com(\hH) \tensor \Com(L^2(G))
		\\
		&\subseteq \aA \tensor \Com(\hH) \tensor \Com(L^2(G),\hH).
	\end{align*}
	As $(u_g \tensor \one_\hH  \tensor \one_\hH) \tilde{S} x= \tilde{S} x (u_g \tensor \one_\hH  \tensor \one_G)$ for all $g \in G$ and $x \in \aD \tensor \Com(L^2(G))$, which is due to the second part of the lemma and the fact that $\aD$ is fixed by $\Ad[u_g \tensor \one_\hH]$, $g \in G$, we further conclude that
	\begin{align*}
		\tilde{S} \aD \tensor \Com(L^2(G)) \subseteq \aB \tensor \Com(\hH) \tensor \Com(L^2(G),\hH),
	\end{align*}
	and hence that
	\begin{align*}
		\tilde{S} \aD \tensor \Com(L^2(G)) &= (P)_{12} \tilde{S} \aD \tensor \Com(L^2(G)) (P)_{12}
		\\
		&\subseteq (P)_{12} \aB \tensor \Com(\hH) \tensor \Com(L^2(G),\hH) (P)_{12} 
		= \aD \tensor \Com(L^2(G),\hH).
	\end{align*}
	In the same manner we can see that $\aD \tensor \Com(\hH) \tilde{S} \subseteq \aD \tensor \Com(L^2(G),\hH)$.
	From this it follows that
$\tilde{S} \in \Mul(\aD \tensor \Com(L^2(G) \oplus \hH))$ with $(1_{\hH_P} \tensor p_{L^2(G)}) S = 0 =S (1_{\hH_P} \tensor p_\hH)$ as claimed.				
\item
	This is clear from Lemma~\ref{lem:multipliers}, because $\tilde{S} \Mul(D)  \tensor \Com(L^2(G)) \subseteq \End(\hH_P) \tensor \Com(L^2(G) \oplus \hH)$ and 
	\begin{align*}
		\tilde{S} \Mul(D) \tensor \Com(L^2(G)) &\subseteq \Mul(\aD \tensor \Com(L^2(G) \oplus \hH)) \Mul(\aD \tensor \Com(L^2(G) \oplus \hH))
		\\
		&\subseteq \Mul(\aD \tensor \Com(L^2(G) \oplus \hH))
	\end{align*}
	by the third part of the lemma.
	\qedhere
\end{enumerate}
\end{proof}

\begin{corollary}\label{cor:reconstruction}
	The C\Star-dynamical system $(\tilde{\aA},G,\tilde{\alpha})$ is free.
\end{corollary}
\begin{proof}
	Our strategy of proof is to apply Lemma~\ref{lem:coisometry}.
	Due to the first two statements of Lemma~\ref{lem:reconstruction}, we can already assert that $\tilde{S} \in \End(\hH_P \tensor L^2(G), \hH_P \tensor \hH)$ is an isometry satisfying~\eqref{eq:SOPequivariance} and~\eqref{eq:SOPcommuting}.
	What is left is to establish that $\tilde{S} \tilde{\aA} \tensor \Com(L^2(G)) \subseteq \tilde{\aA} \tensor \Com(L^2(G),\hH)$. 
	For this purpose, let $x \in \tilde{\aA} \tensor \Com(L^2(G))$. 
	By Lemma~\ref{lem:reconstructionSUBMULTI}, we have $\tilde{S} x \in \Mul(\aD) \tensor \Com(L^2(G),\hH)$, and hence we may apply the map $\delta_\aD \tensor \id$ which gives $\delta_\aD \tensor \id (\tilde{S} x) = \delta_\aD \tensor \id (\tilde{S}) \delta_\aD \tensor \id (x) = \tilde{S}_{124} x_{124} = (\tilde{S} x)_{124}$ when combined with Lemma~\ref{lem:reconstructionCOISO}
	That is, $\tilde{S} x \in \fix(\delta_\aD) \tensor \Com(L^2(G),\hH)$.
	To see that, in fact, $\tilde{S} x \in \tilde{\aA} \tensor \Com(L^2(G),\hH)$, we prove that the map $G \ni g \mapsto \Ad [\one_\aA \tensor \mu_g] \tensor \id (\tilde{S} x) \in \fix(\delta_\aD) \tensor \Com(L^2(G),\hH)$ is norm continuous.
	For the latter assertion, we use the identity
	\begin{align*}
		\Ad [\one_\aA \tensor \mu_g] \tensor \id (\tilde{S} x) = \tilde{S} (\one_\aA \tensor \one_\hH  \tensor r_g) \Ad [\one_\aA \tensor \mu_g \tensor \one_G](x),
		&&
		g \in G,
	\end{align*} 
	which is a consequence of Lemma~\ref{lem:reconstructionEQUiV}, together with the immediate norm continuity of the maps $G \ni g \mapsto \Ad [\one_\aA \tensor \mu_g \tensor \one_G](x)$ and $G \ni g \mapsto (\one_\aA \tensor \one_\hH  \tensor r_g) y$, $y \in \End(\hH_P) \tensor \Com(L^2(G))$.
\end{proof}

The fixed point algebra of $(\tilde{\aA},G,\tilde{\alpha})$ is characterized by our next result.

\begin{lemma}\label{lem:fixgamma}
	The fixed point algebra of the C\Star-dynamical system $(\tilde{\aA},G,\tilde{\alpha})$ is the image of the faithful unital \Star-homomorphism $\gamma : \aB \rightarrow \tilde{\aA}$, $\gamma(b) := S(b \tensor \one_G) S^*$.
\end{lemma}
\begin{proof} 
	We first establish that $\gamma$ is well-defined. 
	To do this, let $b \in \aB$ and let $x \in \aD$. 
	By~\eqref{eq:SOPleft} and~\eqref{eq:SOPright} we have $\gamma(b)x \in \aA \tensor \Com(\hH)$. 
	Combining~\eqref{eq:SOPequivariance} with the fact that $\aD$ is fixed under $\Ad[u_g \tensor \one_\hH]$, $g \in G$, further yields $\gamma(b)x \in \aB \tensor \Com(\hH)$
	Hence $\gamma(b)x \in \aD$, because $\gamma(b)x = S^*S \gamma(b)x S^*S$.
	In the same manner we see that $x\gamma(b) \in\aD$.
	In consequence, $\gamma(b) \in \Mul(\aD)$. 
	Since
	\begin{align*}
		\delta_\aD(\gamma(b)) &= \Ad[ S_{12} (W_G)_{23} S_{12}^* ] ( S_{12} ( b \tensor \one_G \tensor \one_G ) S_{12}^* )
		\\
		&= \Ad[ S_{12} (W_G)_{23} S_{12}^* S_{12} ] ( b \tensor \one_G \tensor \one_G) 
		\\
		&= \Ad[ S_{12}] (b \tensor \one_G \tensor \one_G) = \gamma(b) \tensor \one_G,
	\end{align*}
	\ie, $\gamma(b) \in \fix(\delta_\aD)$, and $\Ad[\one_\aA \tensor \mu_g](\gamma(b)) = \gamma(b)$ for all $g \in G$, which is due to~\eqref{eq:SOPcommuting}, it may finally be concluded that $\gamma(b) \in \tilde{\aA}$. 
	Note that we have actually proved more, namely that $\gamma(b)$ is fixed under $\tilde{\alpha}_g$, $g \in G$, which amounts to saying that $\gamma(\aB) \subseteq \tilde{\aA}^G$.
	We thus proceed to show that $\tilde{\aA}^G \subseteq \gamma(\aB)$. 
	For this, let $x \in \tilde{\aA}^G$.
	We put $y := S^* x S$ and observe that $(W_G)_{23} (y \tensor \one_G) = (y \tensor \one_G) (W_G)_{23}$ or, equivalently, $\id \tensor \delta_G (y \tensor \one_G) = y \tensor \one_G$. 
	As the coaction $\delta_G$ is ergodic, we can assert that $y = b \tensor \one_G$ for some $b \in \aB$. 
	It follows that $x = P x P= S y S^* = S (b \tensor \one_G) S^* = \gamma(b)$, \ie, $\tilde{\aA}^G \subseteq \gamma(\aB)$ as required. 
	That $\gamma$ is a faithful unital \Star-homomorphism is clear, and so the proof is complete.
\end{proof}

We proceed with a technical no-go result which is also of independent interest.
For its proof we make use of the fact that each C\Star-dynamical system $(\aA,G,\alpha)$ can be decomposed into its isotypic components, let's say, $A(\sigma)$, $\sigma \in \Irrep(G)$, which amounts to saying that their algebraic direct sum forms a dense \Star-subalgebra of~$\aA$ (see, \eg, \cite[Thm.~4.22]{HoMo06}). 

\begin{lemma}\label{lem:autosurj}
	Let $(\aA,G,\alpha)$ be a free C\Star-dynamical system with fixed point algebra $\aB$. 
	Furthermore, let $\aA_0$ be a $G$-invariant unital C\Star-subalgebra. 
	If the induced C\Star-dynamical system $(\aA_0,G,\alpha)$ is free and $\aA_0^G=\aB$, then $\aA_0 = \aA$.
\end{lemma}
\begin{proof}
	We shall have established the lemma if we prove that the respective isotypic components are equal, \ie, $A_0(\sigma) = A(\sigma)$ for all $\sigma \in \Irrep(G)$.
	To prove the latter assertion, we fix $\sigma \in \Irrep(G)$ and recall that the freeness of $(\aA,G,\alpha)$ implies that $A(\sigma)$ is a Morita equivalence bimodule between the unital C\Star-algebras $\aB \tensor \End(V_\sigma)$ and 
	\begin{equation*}
		\aC := \{ x \in \aA \tensor \End(V_\sigma) : (\forall g \in G) \, (\one_\aA \tensor \sigma_g) \alpha_g(x) = x (\one_\aA \tensor \sigma_g) \}
	\end{equation*}
	(\cf~\cite[Sec. 3]{SchWa15}). 
	Likewise, the freeness of $(\aA_0,G,\alpha)$ implies that $A_0(\sigma)$ is a Morita equivalence bimodule between the unital C\Star-algebras $\aB \tensor \End(V_\sigma)$ and 
	\begin{equation*}
	 	\aC_0 := \{ x \in \aA_0 \tensor \End(V_\sigma) : (\forall g \in G) \, (\one_\aA \tensor \sigma_g) x = x (\one_\aA \tensor \sigma_g)\}.
	\end{equation*} 
	This makes it possible to apply \cite[Lem.~A.1]{SchWa17} which gives elements $s_1,\ldots,s_n \in A_0(\sigma)$ such that $\sum_{k=1}^n \rprod{\alg C_0}{s_k,s_k} = \one_{\aC}$.
	Since the inner products of $A_0(\sigma)$ are exactly the restrictions in domain and codomain of the respective inner products of $A(\sigma)$, we conclude that $\rprod{\aC_0}{s_k,s_k} = \rprod{\aC}{s_k,s_k}$ for all $1 \leq k \leq n$. 
	Hence for each $x \in A(\sigma)$ we have
	\begin{align*}
		x 
		= x \acts \one_{\aC} = \sum_{k=1}^n x \acts \rprod{\aC_0}{s_k,s_k} 
		= \sum_{k=1}^n x \acts \rprod{\aC}{s_k,s_k} 
		= \sum_{k=1}^n \lprod{\aB \tensor \End(V_\sigma)}{x,s_k} \acts s_k \in A_0(\sigma),
	\end{align*}
	\ie, $A(\sigma) \subseteq A_0(\sigma)$.
	This completes the proof, because the other inclusion is obvious.
\end{proof}

Now, we are almost in a position to state and prove the main result of this section.
The only preparatory point remaining concerns the following map:
\begin{align}
	j_\alpha : \aA \rightarrow \Cont(G,\aA) = \aA \tensor \Cont(G) \subseteq \End(\hH_\aA \tensor L^2(G)),
	&&
	j_\alpha(x)(g) = \alpha_{g^{-1}}(x) = u_g^* x u_g.
	\label{eq:jalpha}
\end{align}
It is clear that $j_\alpha$ is injective. 
Moreover, straightforward computations reveal that
\begin{align}
	\Ad[u_g \tensor r_g] \circ j_\alpha 
	&= j_\alpha  
	&& 
	\forall g \in G
	\label{eq:jalphacommuting},
	\\
	\Ad[\one_\aA \tensor \lambda_g] \circ j_\alpha 
	&= j_\alpha \circ \alpha_g  
	&& 
	\forall g \in G
	\label{eq:jalphaequivariance}.
\end{align}

\begin{theorem}\label{thm:pi_S}
	The map
	\begin{align}
		\pi_S : \aA \rightarrow \tilde{\aA} \subseteq\Mul(\aD) \subseteq \End(\hH_P),
		&&
		\pi_S(x) := S j_\alpha(x) S^* \label{eq:pi_S}
	\end{align}
	is a $G$-equivariant \Star-isomorphism.
\end{theorem}
\begin{proof}
	We start again by proving that the map under consideration is well-defined. 
	For this purpose, let $x \in \aA$. 
	The verification of $\pi_S(x) \in \Mul(\aD)$ can be handled in much the same way as in the proof of Lemma~\ref{lem:fixgamma}, the only difference being that \eqref{eq:SOPequivariance} needs to be combined with \eqref{eq:jalphacommuting} to establish that $\pi_S(x) d, d \pi_S(x) \in \aB \tensor \Com(\hH)$ for all $d \in \aD$.
	Moreover, the identity $(W_G)_{23} j_\alpha(x)_{12} (W_G)_{23}^* = j_\alpha(x)_{12}$, which is easy to check, implies that
	\begin{align*}
		\delta_\aD(\pi_S(x)) &= \Ad[S_{12} (W_G)_{23} S_{12}^*] (S_{12} j_\alpha(x)_{12} S_{12}^*) 
		\\
		&= \Ad[S_{12}] ((W_G)_{23} j_\alpha(x)_{12} (W_G)_{23}^*) 
		\\
		&= \Ad[S_{12}] (j_\alpha(x)_{12})  = \pi_S(x)_{12},
	\end{align*}
	\ie, $\pi_S(x) \in \fix(\delta_\aD)$. 
	That, in fact, $\pi_S(x) \in \tilde{\aA}$ as claimed now follows from
	\begin{align}
		 \tilde{\alpha}_g (\pi_S(x)) 
		 \overset{\eqref{eq:SOPcommuting}}= 
		 S (\Ad[\one_\aA \tensor \lambda_g]( j_\alpha(x)) S^*
		 \overset{\eqref{eq:jalphaequivariance}}= 
		 \pi_S(\alpha_g(x)),
		 &&
		 g \in G.
		 \label{eq:pi_Sequivariant}
	\end{align}
	Of course,~\eqref{eq:pi_Sequivariant} also establishes that $\pi_S$ is $G$-equivariant, and so it remains to show that $\pi_S$ is a \Star-isomorphism. 
	Clearly, $\pi_S$ is a \Star-homomorphism. 
	Moreover, it is injective, because $j_\alpha$ is in jective and $S$ is an isometry.
	With this, we can assert that the induced C\Star-dynamical system $(\pi_S(\aA),G,\tilde{\alpha})$ is free with fixed point algebra $\pi_S(\aB) = \gamma(\aB) = \tilde{\aA}^G$ (\cf Lemma~\ref{lem:fixgamma} for the latter equalities). 
	Hence Lemma~\ref{lem:autosurj} implies that $\pi_S$ is surjective, \ie, $\pi_S(\aA) = \tilde{\aA}$, and therefore the proof is complete.
\end{proof}

\begin{remark}
	An alternative strategy for proofing Theorem~\ref{thm:pi_S} would be to establish that the isometries $S$ and $\tilde{S}$ are conjugated in the sense of~\cite[Thm.~4.4]{SchWa17}. 
	With some technical effort this can be done.
\end{remark}

\begin{corollary}[{\cf~\cite[Cor.~4.2]{SchWa21}}]\label{cor:pi_S}
	The pair $(\pi_S, \Ad[\one_\aA \tensor \mu])$ is a faithful generalized covariant representation of $(\aA,G,\alpha)$ on $\Mul(\aD) = P\Mul(\aB \tensor \Com(\hH)) P$. 
	Moreover, any faithful \Star-representation $\pi_\aB : \aB \to \End(\hH_\aB)$ gives rise to an honest faithful covariant representation on $P(\hH_\aB \tensor \hH)$.
\end{corollary}

\begin{corollary}[{\cf~\cite[Thm.~10]{Wass89}}]\label{cor:cleft}
	Let $(\aA,G,\alpha)$ be a cleft C\Star-dynamical system (\cf~Remark~\ref{rem:cleft}).
	Then $(\aA,G,\alpha)$ can be realized as the invariants of an equivariant coaction of $G$ on $\aB \tensor \Com(L^2(G))$.
	Moreover, any faithful \Star-representation $\pi_\aB : \aB \to \End(\hH_\aB)$ gives rise to an honest faithful covariant representation on $\hH_\aB \tensor L^2(G)$.
\end{corollary}


\bibliographystyle{abbrv}
\bibliography{short,RS}

\begin{thebibliography}{10}

\bibitem{AnGuIsRu22}
P.~Antonini, D.~Guido, T.~Isola, and A.~Rubin.
\newblock A note on twisted crossed products and spectral triples.
\newblock {\em J. Geom. Phys.}, 180:104640, 2022.

\bibitem{BaCoHa15}
P.~Baum, K.~D. Commer, and P.~M. Hajac.
\newblock Free actions of compact quantum groups on unital {C$^*$}-algebras.
\newblock {\em Doc. Math.}, 22:825--849, 2017.

\bibitem{BB12}
B.~Blackadar.
\newblock Operator algebras theory of {C$^*$}-algebras and von neumann
  algebras.
\newblock Lecture notes, May 2012.

\bibitem{CacMes19}
B.~{\'C}a{\'c}i{\'c} and B.~Mesland.
\newblock Gauge theory on noncommutative {R}iemannian principal bundles.
\newblock {\em Comm. Math. Phys}, 388(1):107--198, 2021.

\bibitem{DaSiZu14}
L.~Dabrowski, A.~Sitarz, and A.~Zucca.
\newblock {Dirac} operators on noncommutative principal circle bundles.
\newblock {\em Int. J. Geom. Methods Mod. Phys.}, 11(1):1450012, 2014.

\bibitem{EKQR06}
S.~Echterhoff, S.~Kaliszewski, J.~Quigg, and I.~Raeburn.
\newblock A categorical approach to imprimitivity theorems for
  {C$^*$}-dynamical systems.
\newblock {\em Mem. Amer. Math. Soc.}, 180(850):viii+169, 2006.

\bibitem{Ell00}
D.~A. Ellwood.
\newblock A new characterisation of principal actions.
\newblock {\em J. Func. Anal.}, 173(1):49--60, 2000.

\bibitem{FaLa24}
C.~Farsi and F.~Latr\'emoli\`ere.
\newblock Collapse in noncommutative geometry and spectral continuity.
\newblock preprint, arXiv:2404.00240v1, Mar. 2024.

\bibitem{GoLaPe94}
E.~C. Gootman, A.~J. Lazar, and C.~Peligrad.
\newblock Spectra for compact group actions.
\newblock {\em J. Operator Theory}, 31(2):381--399, 1994.

\bibitem{HoMo06}
K.~H. Hofmann and S.~A. Morris.
\newblock {\em The structure of compact groups}, volume~25 of {\em De Gruyter
  Studies in Mathematics}.
\newblock De Gruyter, Berlin, 2013.

\bibitem{Lan79}
M.~B. Landstad.
\newblock Duality theory for covariant systems.
\newblock {\em Transactions of the American Mathematical Society},
  248(2):223--267, 1979.

\bibitem{Ped18}
G.~K. Pedersen.
\newblock {\em {C$^*$}-algebras and their automorphism groups}.
\newblock Pure and Applied Mathematics (Amsterdam). Academic Press, London,
  2018.
\newblock Second edition, Edited and with a preface by S\o ren Eilers and Dorte
  Olesen.

\bibitem{Phi87}
C.~N. Phillips.
\newblock {\em Equivariant {K}-Theory and Freeness of Group Actions on
  {C$^*$}-Algebras}, volume 1274 of {\em Lecture Notes in Math.}
\newblock Springer, 1987.

\bibitem{Phi09}
C.~N. Phillips.
\newblock Freeness of actions of finite groups on {C$^*$}-algebras.
\newblock In {\em Operator structures and dynamical systems}, volume 503 of
  {\em Contemp. Math.}, pages 217--257. Amer. Math. Soc., 2009.

\bibitem{Rae98}
I.~Raeburn and D.~Williams.
\newblock {\em Morita Equivalence and Continuous-Trace {C$^*$-algebras}},
  volume~60 of {\em Math. Surv. Monogr.}
\newblock Providence, RI: American Mathematical Society (AMS), 1998.

\bibitem{Rieffel91}
M.~A. Rieffel.
\newblock Proper actions of groups on {C$^*$}-algebras.
\newblock In {\em Mappings of Operator Algebras}, volume~84 of {\em Progr.
  Math.}, pages 141--182. Birkh{\"a}user, 1991.

\bibitem{SchWa15}
K.~Schwieger and S.~Wagner.
\newblock {Part} {I}, {Free} actions of compact {A}belian groups on
  {C$^*$}-algebras.
\newblock {\em Adv. Math.}, 317:224--266, 2017.

\bibitem{SchWa16}
K.~Schwieger and S.~Wagner.
\newblock {Part} {II}, {Free} actions of compact groups on {C$^*$}-algebras.
\newblock {\em J. Noncommut. Geom.}, 11(2):641--688, 2017.

\bibitem{SchWa17}
K.~Schwieger and S.~Wagner.
\newblock {Part III}, {Free} actions of compact quantum groups on
  {C$^*$}-algebras.
\newblock {\em SIGMA}, 13(062):19 pages, 2017.

\bibitem{SchWa21}
K.~Schwieger and S.~Wagner.
\newblock An {Atiyah} sequence for noncommutative principal bundles.
\newblock {\em SIGMA}, 18(015):22 pages, 2022.

\bibitem{SchWa20}
K.~Schwieger and S.~Wagner.
\newblock Lifting spectral triples to noncommutative principal bundles.
\newblock {\em Adv. Math.}, 396:108160, 2022.

\bibitem{Wass88a}
A.~Wassermann.
\newblock Ergodic actions of compact groups an operator algebras. {II}.
  {Classification} of full multiplicity ergodic actions.
\newblock {\em Canad. J. Math.}, 40(6):1482--1527, 1988.

\bibitem{Wass89}
A.~Wassermann.
\newblock Ergodic actions of compact groups on operator algebras. {I}.
  {G}eneral theory.
\newblock {\em Ann. of Math. (2)}, 130(2):273--319, 1989.

\end{thebibliography}

\end{document}